\documentclass[12pt]{article}
\usepackage{amsmath, amsthm}\usepackage{enumerate}\usepackage{amssymb}\usepackage{bbold}\usepackage{color}
\newtheorem{theorem}{Theorem}[section]
\newtheorem{proposition}[theorem]{Proposition}

\newtheorem{example}[theorem]{Example}
\newtheorem{corollary}[theorem]{Corollary}
\newtheorem{remark}[theorem]{Remark}
\newtheorem{observation}[theorem]{Observation}
\newtheorem{definition}[theorem]{Definition}

\DeclareMathOperator{\rc}{\xrightarrow[]{\mathbb{r}}}
\DeclareMathOperator{\oc}{\xrightarrow[]{\mathbb{o}}}
\DeclareMathOperator{\cc}{\xrightarrow[]{\mathbb{c}}}
\DeclareMathOperator{\mrc}{\xrightarrow[]{\mathbb{mr}}}
\DeclareMathOperator{\moc}{\xrightarrow[]{\mathbb{mo}}}

\DeclareMathOperator{\lmoc}{\xrightarrow[]{\mathbb{m}_l\mathbb{o}}}
\DeclareMathOperator{\rmoc}{\xrightarrow[]{\mathbb{m}_r\mathbb{o}}}
\DeclareMathOperator{\lmcc}{\xrightarrow[]{\mathbb{m}_l\mathbb{c}}}
\DeclareMathOperator{\rmccA}{\xrightarrow[]{\mathbb{m}_r\mathbb{c}(A)}}
\DeclareMathOperator{\rmcc}{\xrightarrow[]{\mathbb{m}_r\mathbb{c}}}
\DeclareMathOperator{\mcc}{\xrightarrow[]{\mathbb{m}\mathbb{c}}}

\DeclareMathOperator{\mrrc}{\xrightarrow[]{\mathbb{m}_r\mathbb{r}}}
\DeclareMathOperator{\convcone}{\xrightarrow[]{\mathbb{c}_1}}
\DeclareMathOperator{\convctwo}{\xrightarrow[]{\mathbb{c}_2}}

\begin{document}

\title{Multiplicative order continuous operators on Riesz algebras}
\maketitle
\author{\centering {Abdullah Ayd{\i}n$^{1}$, Eduard Emelyanov$^{2}$, Svetlana Gorokhova$^{3}$}\\ 
\small $^1$ Mu\c{s} Alparslan University,  Turkey, a.aydin@alparslan.edu.tr\\
\small $^2$ Sobolev Institute of Mathematics, Russia, emelanov@math.nsc.ru\\ 
\small $^3$ Southern Mathematical Institute, Russia, lanagor71@gmail.com\\ 
}
\begin{abstract}
In this paper, we investigate operators on Riesz algebras, which are continuous with respect to multiplicative modifications of  order convergence and relatively uniform convergence. We also introduce and study mo-Lebesgue, mo-$KB$, and mo-Levi operators.
\end{abstract} 


\maketitle

\section{Introduction}

It is well known that the order, relatively uniform, unbounded order, and various multiplicative 
order convergences in Riesz algebras are not topological in general (see for example \cite[Thm.5]{DEM1}, \cite[Thm.2.2]{DEM2}, \cite[Thm.2]{Gor}). 
However, even without topological structure, several kinds of continuous operators can be defined (see, e.g. \cite{AEEM1}). 
As far as we know, there is no sufficiently  comprehensive study of operator theory on Riesz algebras. 
The aim of this paper is to put forth and study operators on Riesz algebras which are continuous e.g. with respect to 
$\mathbb{mo}$- or $\mathbb{mr}$- convergences. The results obtained in the settings of Riesz algebras may also shed light 
on operators on the lattice-normed algebras (cf. \cite{Ay1,Ay2,AEEM1,AEEM2,Pliev}). 
Throughout the paper we assume that all vector spaces are real and all operators are linear. In what follows, we denote by letters $X$ and $Y$ Riesz spaces and  by $E$ and $F$ Riesz algebras.

A net $(x_\alpha)_{\alpha\in A}$ in a Riesz space $X$ is {\em order convergent} 
(or {\em $\mathbb{o}$-convergent}, for short) to $x\in X$, if there exists a net $(y_\beta)_{\beta\in B}$ satisfying 
$y_\beta \downarrow 0$, 
and for any $\beta\in B$, there exists $\alpha_\beta\in A$, such that $|x_\alpha-x|\leq y_\beta$ for all $\alpha\geq\alpha_\beta$. 
In this case, we write $x_\alpha\oc x$. 
A net $(x_\alpha)_{\alpha\in A}$ in $X$ {\em relatively uniform converges} to $x\in X$
($x_{\alpha}\rc x$, for short) if there exists $u\in X_+$, such that, for any $n\in\mathbb{N}$, 
there exists $\alpha_n$ such that $|x_\alpha-x|\le\frac{1}{n} u$ for all $\alpha\ge\alpha_n$ 
(see, for example, \cite[1.3.4, p.20]{Ku}). An operator $T$ between Riesz spaces is called:
\begin{enumerate}
\item[-] {\em order bounded}, if $T$ takes order bounded sets to order bounded sets;  
\item[-] {\em regular}, if $T=T_1-T_2$ with $T_1,T_2\ge 0$;
\item[-] {\em order continuous}, if $Tx_\alpha \oc Tx$ whenever $x_\alpha\oc x$;
\item[-] {\em relatively uniform continuous}, if $Tx_\alpha \rc Tx$ whenever $x_\alpha\rc x$.
\end{enumerate}
It is well known that order continuous and relatively order continuous operators are order bounded.
The collection ${\cal L}_r(X,Y)$ of all regular operators between Riesz spaces $X$ and $Y$
is a subspace of the vector space ${\cal L}_b(X,Y)$ of all bounded operators from $X$ to $Y$.
Let $Y$ be Dedekind complete, then ${\cal L}_r(X,Y)$ is a Dedekind complete Riesz space (cf. \cite[Thm.1.67]{AB1})  containing 
the collection ${\cal L}_n(X,Y)$ of all order continuous operators from $X$ to $Y$ as a band (cf. \cite[Thm.1.73]{AB1}).
We write ${\cal L}_r(X)$ for ${\cal L}_r(X,X)$; ${\cal L}_n(X)$ for ${\cal L}_n(X,X)$; etc.

Recall that a Riesz space $E$ is called {\em Riesz algebra} 
if $E$ is an associative algebra whose positive cone $E_+$ is closed under the algebra multiplication, i.e., $x\cdot y\in E_+$ 
whenever $x,y\in E_+$. A Riesz algebra $E$ is called (cf. \cite{AB2,AEG,BR,Hu,Ku,Pag}): 
\begin{enumerate}
\item[-] {\em left $($right$)$ $d$-algebra}, whenever $u\cdot(x\wedge y)=(u\cdot x)\wedge(u\cdot y)$ 
(resp., $(x\wedge y)\cdot u=(x\cdot u)\wedge(y\cdot u)$) for all  $x,y\in E$ and $u\in E_+$;
\item[-] {\em $d$-algebra} if $E$ is both left and right $d$-algebra;
\item[-] {\em left $($right$)$ $f$-algebra} if $x\wedge y=0$ implies $(u\cdot x)\wedge y=0$ 
(resp., $(x\cdot u)\wedge y=0$)  for all $u\in E_+$;
\item[-] {\em $f$-algebra} if $E$ is both left and right $f$-algebra;
\item[-] {\em semiprime} whenever the only nilpotent element in $E$ is $0$; 
\item[-] {\em unital} if $E$ has a positive multiplicative unit.
\end{enumerate}
Any Riesz space $E$ becomes a commutative $f$-algebra with respect to 
the {\em trivial algebraic multiplication} $x\ast y = 0$ for all $x,y \in E$,  
and $(E,\ast)$ is neither unital nor semiprime unless $\dim(E)=0$.

\begin{example}\label{l-algebra $R^D$}{\em
\cite[Ex.5]{AEG}. In the Riesz space $\mathbb{R}^D$ of all $\mathbb{R}$-valued functions on a set $D$, 
any $f$-algebra multiplication $\ast$ is uniquely determined by the function $\zeta(d):=[\mathbb{I}_{\{d\}}\ast\mathbb{I}_{\{d\}}](d)$, 
where $\mathbb{I}_{A} \in\mathbb{R}^D$ is the characteristic function of $A\subseteq D$. 
Furthermore, $(\mathbb{R}^D,\ast)$ is unital iff it is semiprime iff $\zeta$ is a weak order unit in $\mathbb{R}^D$.}
\end{example}

Let $\mathbb{c}$ be a linear convergence on a Riesz algebra $E$ 
(see, e.g., \cite[Def.1.6]{AEG}).

\begin{definition}
{\em [cf. \cite[Def.5.3]{AEG}] The algebra multiplication in $E$ is called {\em left $\mathbb{c}$-continuous} 
{\em  $($right $\mathbb{c}$-continuous}$)$ whenever $x_\alpha\cc x$ implies 
$y\cdot x_\alpha\cc y\cdot x$ (resp., $x_\alpha\cdot y\cc x\cdot y$) for every $y\in E$.
The algebra multiplication is called {\em $\mathbb{c}$-continuous} if it is both  left and right $\mathbb{c}$-continuous.} 
\end{definition}

\begin{definition}
{\em Let $X$ and $Y$ be Riesz spaces equipped with linear convergences $\mathbb{c}_1$ and $\mathbb{c}_2$ respectively.
An operator $T:X\to Y$ is called {\em $\mathbb{c}_1\mathbb{c}_2$-continuous}, whenever $x_\alpha\convcone x$ in $X$ implies $Tx_\alpha\convctwo Tx$ in $Y$. 
In the case when $\mathbb{c}_1 = \mathbb{c}_2$, we say that $T$ is {\em $\mathbb{c}_1$-continuous}. 
The collection of all $\mathbb{c}_1\mathbb{c}_2$-continuous operators from $X$ to $Y$ is denoted by 
${\cal L}_{\mathbb{c}_1\mathbb{c}_2}(X,Y)$, and if $\mathbb{c}_1 = \mathbb{c}_2$, we denote 
${\cal L}_{\mathbb{c}_1\mathbb{c}_2}(X,Y)$ by ${\cal L}_{\mathbb{c}_1}(X,Y)$, and 
${\cal L}_{\mathbb{c}_1}(X,X)$ by ${\cal L}_{\mathbb{c}_1}(X)$.}
\end{definition}

A net $x_\alpha$ in a Riesz algebra $E$ is said to be {\em  left $($right$)$ multiplicative $\mathbb{c}$-convergent} to $x$ whenever
$$
   u\cdot |x_\alpha-x|\cc 0 \ \ \ \ (\text{respectively} 
   \ \ \  |x_\alpha-x|\cdot u\cc 0) \ \ \ (\forall u\in E_+),
$$	
briefly $x_\alpha\lmcc x$ (resp. $x_\alpha\rmcc x$).
If  $x_\alpha\lmcc x$ and $x_\alpha\rmcc x$ simultaneously, we write $x_\alpha\mcc x$ (cf. \cite[Def.5.4]{AEG}, \cite{Ay1,Ay2,AE}).
It is worth to notice that many multiplicative $\mathbb{c}$-convergences could be defined by specifying the collections of admissible
factors, e.g. the right multiplicative $\mathbb{c}$-convergence of a net $x_\alpha$ in $E$ to $x\in E$ w.r. to  $A\subseteq E$:
$$
    x_\alpha\rmccA x \ \ \ \text{whenever} \ \ \  |x_\alpha-x|\cdot u\cc 0 \ \ \ (\forall u\in A)
$$
We postpone the study of such specified convergences to further papers.  

\begin{remark}\label{the same}
{\em Clearly,  $\mathbb{m}_l\mathbb{c} \equiv \mathbb{m}_r\mathbb{c}$ in commutative algebras. Moreover, 
$\mathbb{m}_l\mathbb{c}$-convergence turns to $\mathbb{m}_r\mathbb{c}$-convergence and vice versus, 
if we replace the algebra multiplication ``~$\cdot$~" in $E$ by ``~$\hat{\cdot}$~" such that $x\ \hat{\cdot}\ y := y\cdot x$.}
\end{remark}

We refer for more information on Riesz spaces and Riesz algebras to \cite{AB1,AB2, AGG, AEEM2, AEG,  Ku,LZ,Pag,Vu,Za}. 
The structure of the present paper is as follows. In Section 2, we investigate $\mathbb{mo}$-, $\mathbb{rmo}$, and 
$\mathbb{omo}$-continuous operators. Section 3 is devoted to $\mathbb{mo}$-Lebesgue, $\mathbb{mo}$-$KB$, 
and $\mathbb{mo}$-Levi operators, and to the domination problem for such operators. 

\section{Basic properties of $\mathbb{mo}$-continuous operators}

In this section, we investigate basic properties of $\mathbb{mo}$-, $\mathbb{omo}$-, and $\mathbb{rmo}$-continuous operators 
in  Riesz algebras.

Let $X$ be a Dedekind complete Riesz space. The space ${\cal L}_r(X)$ is a unital Dedekind complete 
Riesz algebra under the composition operation, and the space ${\cal L}_n(X)$ is a Riesz subalgebra of ${\cal L}_r(X)$.
It is well known that the algebra multiplication in ${\cal L}_r(X)$ is right $\mathbb{o}$-continuous, whereas in 
${\cal L}_n(X)$ it is $\mathbb{o}$-continuous (see, e.g. \cite[Thm.1.56]{AB2}).
The following result extends this fact to $\mathbb{m}\mathbb{o}$-convergence.

\begin{theorem}\label{main example}
Let $X$ be a Dedekind complete Riesz space. Then the algebra multiplication is$:$
\begin{enumerate}
\item[$(i)$] right $\mathbb{m}_r\mathbb{o}$-continuous in $E={\cal L}_r(X)$$;$
\item[$(ii)$] left and right  $\mathbb{m}\mathbb{o}$-continuous in $F={\cal L}_n(X)$.
\end{enumerate}
\end{theorem}

\begin{proof}
$(i)$ Let $T_\alpha\rmoc T$ in $E$ and $R\in E$, Then, for every $S\in E_+$, 
$$
   |T_\alpha\circ R-T\circ R|\circ S\le |T_\alpha-T|\circ |R|\circ S=|T_\alpha-T|\circ(|R|\circ S)\oc 0.
   \eqno(1)
$$
By $($1$)$,  $|T_\alpha\circ R-T\circ R|\circ S\oc 0$ for every $S\in E_+$ and 
hence $T_\alpha\circ R\rmoc T\circ R$. Since $R\in E$ is arbitrary, the algebra multiplication  in $E$
is right $\mathbb{m}_r\mathbb{o}$-continuous.

$(ii)$
Let $T_\alpha\moc T$ in $F$ and $R\in F$. Since $S \circ|T_\alpha-T|\oc 0$ for each $S\in F_+$, it follows from 
$$
   S \circ|T_\alpha\circ R-T\circ R|\le S \circ|T_\alpha-T|\circ |R|=(S \circ|T_\alpha-T|)\circ|R| \ \ \ (\forall S\in F_+)
$$
that $S\circ|T_\alpha\circ R-T\circ R|\oc 0$, because the multiplication in $E$ is right $\mathbb{o}$-continuous 
by \cite[Thm.1.56]{AB2}. Since $S\in F$  is arbitrary, we conclude that $T_\alpha\circ R\lmoc T\circ R$ and hence
the multiplication in $F$ is right $\mathbb{m}_l\mathbb{o}$-continuous. 
It follows from $(i)$ that $T_\alpha\circ R\rmoc T\circ R$, and hence $T_\alpha\circ R\moc T\circ R$.
Since $R\in F$ is arbitrary, the algebra multiplication in $F$ is right $\mathbb{mo}$-continuous.

We skip similar elementary arguments which shows that the algebra multiplication in $F$ is also left $\mathbb{mo}$-continuous.
\end{proof}

The following example shows that, in general, the algebra multiplication is not left $\mathbb{m}_r\mathbb{o}$-continuous in 
${\cal L}_r(X)$.

\begin{example}\label{not right mo-cont}
{\em Take a free ultrafilter $\cal{U}$ on $\mathbb{N}$. Recall that any bounded real sequence $x_k$ {\em converges along} 
$\cal{U}$ to some $x_{\cal{U}}=\lim_{\cal{U}}x_k \in\mathbb{R}$ in the sense that $\{k\in\mathbb{N}:|x_k-x_{\cal{U}}|\le\varepsilon\}\in\cal{U}$ 
for every $\varepsilon>0$. Define operators 
$L,T_n\in {\cal L}_r(\ell^\infty)$ by $Lx:=x_{\cal{U}}\cdot \mathbb{1}_\mathbb{N}$
and $T_nx:=x_{\cal{U}}\cdot\mathbb{1}_{\{k\in\mathbb{N}:k\ge n\}}$.
It is easy to see that $T_n\downarrow 0$ and hence $T_n\rmoc 0$
in ${\cal L}_r(\ell^\infty)$.
However, 
$$
 L\circ T_n(x)= L(x_{\cal{U}}\cdot \mathbb{1}_{\{k\in\mathbb{N}:k\ge n\}})
                      = x_{\cal{U}}\cdot\mathbb{1}_\mathbb{N}
$$ 
implies that
$L\circ T_n\circ I=L\circ T_n \equiv L\ne 0$.
 Thus the sequence  $L\circ T_n$ does not $\mathbb{m}_r\mathbb{o}$-converge to $0$, and hence the algebra multiplication in 
the Riesz algebra ${\cal L}_r(\ell^\infty)$ is neither left $\mathbb{o}$-, nor left $\mathbb{m}_r\mathbb{o}$-continuous.}
\end{example}

In the following straightforward properties, we restrict ourselves to the "right" case only skipping the analogical "left" case.

\begin{observation}\label{remark}
 {\em Let $E$ and $F$ be Riesz algebras. 
\begin{enumerate}[$a)$]
\item The collections ${\cal L}_{\mathbb{m}_r\mathbb{o}}(E,F)$, ${\cal L}_{\mathbb{o}\mathbb{m}_r\mathbb{o}}(E,F)$, and ${\cal L}_{\mathbb{r}\mathbb{m}_r\mathbb{o}}(E,F)$
of all $\mathbb{m}_r\mathbb{o}$-, $\mathbb{o}\mathbb{m}_r\mathbb{o}$-, and $\mathbb{r}\mathbb{m}_r\mathbb{o}$-continuous operators from $E$ to $F$ are vector spaces.

\item If $E$ is unital, then ${\cal L}_{\mathbb{o}\mathbb{m}_r\mathbb{o}}(E,F) \subseteq {\cal L}_{\mathbb{m}_r\mathbb{o}}(E,F)$.

\item If $F$ is unital, then ${\cal L}_{\mathbb{o}\mathbb{m}_r\mathbb{o}}(E,F) \subseteq {\cal L}_n(E,F)$.

\item If $E$ has right $\mathbb{o}$-continuous algebra multiplication then 
${\cal L}_{\mathbb{m}_r\mathbb{o}}(E,F)\subseteq{\cal L}_{\mathbb{o}\mathbb{m}_r\mathbb{o}}(E,F)$ (cf. \cite[Lem.5.5]{AEG}).

\item If $E$ is Archimedean then ${\cal L}_{\mathbb{o}\mathbb{m}_r\mathbb{o}}(E,F) \subseteq {\cal L}_{\mathbb{r}\mathbb{m}_r\mathbb{o}}(E,F)$.

\item If $E$ is Archimedean and has right $\mathbb{o}$-continuous algebra multiplication then
${\cal L}_{\mathbb{m}_r\mathbb{o}}(E,F) \subseteq{\cal L}_{\mathbb{r}\mathbb{m}_r\mathbb{o}}(E,F)$ by $d)$ and $e)$.

\item If $F$ is an Archimedean $f$-algebra then the algebra multiplication in $F$ is commutative and $\mathbb{o}$-continuous (cf. \cite{Pag,Hu}) and hence
${\cal L}_n(E,F) \subseteq {\cal L}_{\mathbb{o}\mathbb{m}\mathbb{o}}(E,F)$ (cf. \cite[Lem.5.1]{AEG}).
\end{enumerate}}
\end{observation}

The next result generalizes \cite[Prop.5.1]{AEG} with essentially the same proof.

\begin{proposition}
The algebra multiplication in any Riesz algebra $F$ is both left and right $\mathbb{r}$-continuous.
\end{proposition}

\begin{proof}
Let $x_\alpha\rc x$ in $F$ and $y\in F$. Then there exists $v\in F_+$ such that, for any $k\in\mathbb{N}$, 
there is $\alpha_k$ with $|x_\alpha-x|\le\frac{1}{k} v$ for all $\alpha\ge\alpha_k$. Since, for all $\alpha\ge\alpha_k$ we have
$$
   |y\cdot x_\alpha - y\cdot x|\le|y|\cdot|x_\alpha-x|\le \frac{1}{k} |y|\cdot v \ \ \& \ \  |x_\alpha\cdot y - x\cdot y|\le|x_\alpha-x|\cdot|y|\le \frac{1}{k} v\cdot |y|
$$ 
for all $\alpha\ge\alpha_k$, it follows $y\cdot x_\alpha\rc y\cdot x$ and $x_\alpha\cdot y\rc x\cdot y$. 
The result follows because of $y\in F$ was taken arbitrary. 
\end{proof}

\begin{proposition}\label{rr subseteq rmo}
Let $X$ be a Riesz space and $F$ an Archimedean Riesz algebra. 
Then ${\cal L}_{\mathbb{rr}}(X,F)\subseteq{\cal L}_{\mathbb{r}\mathbb{m}\mathbb{o}}(X,F)$.
\end{proposition}

\begin{proof}
Let $T\in{\cal L}_{\mathbb{rr}}(X,F)$. Suppose $x_\alpha\rc x$ in $X$. 
Then $Tx_\alpha\rc Tx$ in $F$. Take $u\in F_+$, such that, for any $n\in\mathbb{N}$, there exists $\alpha_n$ 
with $|Tx_\alpha-Tx|\le\frac{1}{n} u$ for all $\alpha\ge\alpha_n$. Take any $w\in F_+$. 
Then $|Tx_\alpha-Tx|\cdot w\le\frac{1}{n} u\cdot w$ for all $\alpha\ge\alpha_n$.
Since $F$ is Archimedean, then $\frac{1}{n} u\cdot w\downarrow 0$ and hence $|Tx_\alpha-Tx|\cdot w\oc 0$. 
Since $w\in F_+$ is arbitrary, $Tx_\alpha\rmoc Tx$.
As the proof of $Tx_\alpha\lmoc Tx$ is analogous, we obtain the desired result.
\end{proof}

\begin{theorem}\label{inclusion}
Let $X$ be a Riesz space and let $F$ be an Archimedean Riesz algebra, then
$$
  {\cal L}_r(X,F)\subseteq{\cal L}_{\mathbb{rr}}(X,F)\subseteq
  {\cal L}_{\mathbb{r}\mathbb{m}_l\mathbb{r}}(X,F)\cap {\cal L}_{\mathbb{r}\mathbb{m}_r\mathbb{r}}(X,F)\subseteq 
  {\cal L}_{\mathbb{r}\mathbb{m}\mathbb{o}}(X,F).
$$	
\end{theorem}

\begin{proof}
Despite the first incision is well known,  we include its proof for the convenience.
Let $T\in{\cal L}_r(X,F)$ and $x_\alpha\rc x$ in $X$. Take $T_1,T_2\ge 0$ with $T=T_1-T_2$ and $u\in X_+$, 
so that, for each $n\in\mathbb{N}$, there exists $\alpha_n$ such that $|x_\alpha-x|\le\frac{1}{n} u$ for all $\alpha\ge\alpha_n$. Then 
$$
    |Tx_\alpha-Tx|=|(T_1-T_2)x_\alpha-(T_1-T_2)x|=|T_1(x_\alpha-x)-T_2(x_\alpha-x)|\le
$$
$$
    |T_1(x_\alpha-x)|+|T_2(x_\alpha-x)|\le T_1|x_\alpha-x|+T_2|x_\alpha-x|\le \frac{1}{n}(T_1u+T_2u),
$$
and thus $Tx_\alpha\rc Tx$ in $E$. Hence ${\cal L}_r(X,F)\subseteq {\cal L}_{\mathbb{rr}}(X,F)$.

In the rest of the proof, it suffices to restrict ourselves to the "right" case only.
Let $T\in {\cal L}_{\mathbb{rr}}(X,F)$ and $x_\alpha\rc x$ in $X$. It follows from $Tx_\alpha\rc Tx$
that there exist $w\in F_+$ and a sequence of indexes $\alpha_n$ satisfying $|Tx_\alpha-Tx|\le\frac{1}{n}w$ 
for all $\alpha\ge\alpha_n$. Then, for every $f\in F_+$,
$$
   |Tx_\alpha-Tx|\cdot f\le\frac{1}{n}w\cdot f \ \ \ (\forall \alpha\ge\alpha_n).
\eqno(2)
$$ 
By $($2$)$, $Tx_\alpha\mrrc Tx$ and hence ${\cal L}_{\mathbb{rr}}(X,F)\subseteq {\cal L}_{\mathbb{r}\mathbb{m}_r\mathbb{r}}(X,F)$.

The inclusions ${\cal L}_{\mathbb{r}\mathbb{m}_l\mathbb{r}}(X,F)\subseteq {\cal L}_{\mathbb{r}\mathbb{m}_l\mathbb{o}}(X,F)$ and 
${\cal L}_{\mathbb{r}\mathbb{m}_r\mathbb{r}}(X,F)\subseteq {\cal L}_{\mathbb{r}\mathbb{m}_r\mathbb{o}}(X,F)$ hold true because
the $\mathbb{r}$-convergence implies $\mathbb{o}$-convergence in any Archi\-medean Riesz space, so in $F$.
\end{proof}

\begin{example}
{\em The set $Orth(X)$ of all orthomorphisms on an Archimedean Riesz space $X$ is an Archimedean commutative unital algebra 
with $\mathbb{o}$-continu\-ous algebra multiplication \cite[Thm.8.6]{Pag}. 
Since $\mathbb{mo}$-convergence coincides with $\mathbb{o}$-convergence in $Orth(X)$,
any operator from $Orth(X)$ to an arbitrary Riesz algebra is $\mathbb{mo}$-continuous if and only if it is $\mathbb{omo}$-continuous.}
\end{example}

\begin{example}\label{order bounded not m cont}
{\em \cite[Ex.6]{AEG} Let $\cal{U}$ be a free ultrafilter on $\mathbb{N}$. 
We define an operation $\ast$ in $\ell^\infty$ by 
$x\ast y:=\lim_{\cal{U}}(x_n\cdot y_n)\cdot\mathbb{1}_\mathbb{N}$. 
It is easy to see that $(\ell^\infty,\ast )$ is a commutative $d$-algebra.
The identity operator $I$ on $(\ell^\infty,\ast )$ is order bounded but
it is neither  $\mathbb{om}_l\mathbb{o}$- nor $\mathbb{om}_r\mathbb{o}$-co\-n\-ti\-nuous. 
Indeed, $f_n:=\mathbb{1}_{\{k\in\mathbb{N}:k\ge n\}}\oc 0$ in $\ell^\infty$, yet the sequence
$\mathbb{1} \ast I(f_n) = I(f_n)\ast \mathbb{1}=f_n\ast \mathbb{1}\equiv \mathbb{1}$ $\mathbb{o}$-converges to 
$\mathbb{1}\ne 0 =  \mathbb{1}\ast I(0) =I(0)\ast \mathbb{1}$.}
\end{example}

We continue with an extension of \cite[Thm.12]{AEG}.

\begin{proposition}\label{lr-d}
Let $E$ be a Riesz algebra. The following conditions are equivalent$:$
\begin{enumerate}[$($i$)$]
\item $E$ is a left $($right$)$  $d$-algebra with left $($right$)$ $\mathbb{o}$-continuous multiplication$;$
\item $\inf_E (u\cdot A)=u\cdot\inf_E A$ $($resp., $\inf_E (A\cdot u)=(\inf_E A)\cdot u$$)$ 
for every $u\in E_+$ and $A\subseteq E$ such that $\inf_E A$ exists$;$
\item $\sup_E (u\cdot A)=u\cdot\sup_E A$ $($resp., $\sup_E (A\cdot u)=(\sup_E A)\cdot u$$)$ 
for every $u\in E_+$ and $A\subseteq E$ such that $\sup_E A$ exists.
\end{enumerate}
\end{proposition}
\noindent
The proof of Proposition~\ref{lr-d} for the "right" case is identical with the proof of  Theorem~12 in \cite{AEG}, and the proof for the "left" case is similar.

\begin{theorem}\label{mo implies m}
Each $\mathbb{m}_r\mathbb{o}$-continuous operator from a Riesz algebra $E$ satisfying
$\inf_E (A\cdot u)=(\inf_E A)\cdot u$ for every $u\in E_+$ and $A\subseteq E$, whenever $\inf_E A$ exists, to 
a Riesz algebra $F$ is $\mathbb{o}\mathbb{m}_r\mathbb{o}$-continuous. The same result is true with replacing "right" by "left".
\end{theorem}

\begin{proof}
We restrict ourselves to the "right" case only.
It is enough to show that $\mathbb{o}$-convergence implies $\mathbb{m}_r\mathbb{o}$-convergence in $E$.
Let $x_\alpha\oc x$ in $E$. Then there exists a net $y_\beta$ satisfying 
$y_\beta \downarrow 0$, and for any $\beta$, there is $\alpha_\beta$ such that $|x_\alpha-x|\leq y_\beta$ for all $\alpha\geq\alpha_\beta$. 
Take $u\in E_+$. Then $|x_\alpha-x|\cdot u\leq y_\beta\cdot u$  for all $\alpha\geq\alpha_\beta$. 
It follows from Proposition~\ref{lr-d} $\inf\limits_{\beta\in B}(y_\beta\cdot u)=(\inf\limits_{\beta\in B}y_\beta)\cdot u =0$.
Thus $y_\beta\cdot u\downarrow 0$, and hence $|x_\alpha-x| \cdot u \oc 0$. Since $u\in E_+$ is arbitrary,
 $x_\alpha\rmoc x$.
\end{proof}

Each Archimedean $f$-algebra $E$ is a commutative $d$-algebra (cf. \cite{Hu}), and hence satisfies the conditions of the above theorem due to Proposition~\ref{lr-d}. The following example shows that $\mathbb{o}$-continuity of the algebra multiplication in $E$ is essential in Theorem~\ref{mo implies m}.

\begin{example}\label{Ex A}
{\em Consider the following $d$-algebra multiplication in the Riesz space $c$ of convergent real sequences:
$$
   x\ast y: =(\lim_{n\to\infty} x_n\cdot y_n)\cdot {\mathbb 1}_{\mathbb N}.
   \eqno(3)
$$
The algebra multiplication $\ast$ is not $\mathbb{o}$-continuous since 
$\mathbb{1}_{\{k\ge n\}}\downarrow 0$ but 
$\mathbb{1}_\mathbb{N}\ast\mathbb{1}_{\{k\ge n\}}\equiv\mathbb{1}_\mathbb{N}\ne 0$. 
The identity operator $I$ on $(c,\ast)$ is trivially $\mathbb{mo}$-continuous.
However, it is not $\mathbb{omo}$-continuous: $\mathbb{1}_{\{k\ge n\}}\oc 0$ yet 
$I(\mathbb{1}_{\{k\ge n\}})\ast \mathbb{1}_\mathbb{N}\equiv\mathbb{1}_\mathbb{N}$,
and hence $I(\mathbb{1}_{\{k\ge n\}})$ does not $\mathbb{mo}$-converge to $I(0)=0$.

Furthermore, $I$ is $\mathbb{momr}$-continuous on $(c,\ast)$.
Indeed, $g_\alpha \moc 0$ in $(c,\ast)$ implies $|g_\alpha|\ast\mathbb{1}_\mathbb{N}\oc 0$, and since 
by (3), the sequence  $|g_\alpha|\ast\mathbb{1}_\mathbb{N}$ lies in one dimentional sublattice of $c$, then 
 $|g_\alpha|\ast\mathbb{1}_\mathbb{N}\rc 0$, which implies that
$|g_\alpha|\ast u\rc 0$ for all $u\in c_+$, and hence $g_\alpha\mrc 0$ in $(c,\ast)$.}
\end{example}

\begin{definition}
{\em A subset $A$ of a Riesz algebra $E$ is called:
\begin{enumerate}
\item[-]  {\em $\mathbb{m}_l\mathbb{o}$-bounded} if the set $u\cdot A$ is order bounded for each $u\in E_+$;
\item[-]  {\em $\mathbb{m}_r\mathbb{o}$-bounded} if $A \cdot u$ is order bounded for each $u\in E_+$;
\item[-]  {\em $\mathbb{m}\mathbb{o}$-bounded} if $A$ is both $\mathbb{m}_l\mathbb{o}$- and $\mathbb{m}_r\mathbb{o}$-bounded.
\end{enumerate}}
\end{definition}

Every $\mathbb{o}$-bounded subset of $E$ is $\mathbb{mo}$-bounded; and if the algebra multiplication in $E$ is trivial then every subset of $E$ is $\mathbb{mo}$-bounded. 

\begin{definition}
{\em An operator $T$ from a Riesz space $X$ to a Riesz algebra $F$ is called $\mathbb{m}_l\mathbb{o}$-, 
$\mathbb{m}_r\mathbb{o}$-, or else {\em $\mathbb{m}\mathbb{o}$-bounded} if $T$ maps order bounded subsets of $X$ into 
$\mathbb{m}_l\mathbb{o}$-, $\mathbb{m}_r\mathbb{o}$-, or $\mathbb{m}\mathbb{o}$-bounded subsets of $F$ respectively.}
\end{definition}

If the algebra multiplication in $E$ is trivial, then every operator on $E$ is $\mathbb{mo}$-bounded. 

\begin{remark}
{\em \ Let $X$ be a Riesz space, and let $F$ be a Riesz algebra. Then:
\begin{enumerate}[-]
\item every order bounded $($and hence every positive$)$ operator from $X$ to $F$ is $\mathbb{mo}$-bounded;
\item if $F$ is unital, then every ($\mathbb{m}_l\mathbb{o}$-) $\mathbb{m}_r\mathbb{o}$-bounded operator from $X$ to $F$ is order bounded. 
\end{enumerate}}
\end{remark}

The following theorem is an adopted version of \cite[Thm.2.1]{AS} 
for $\mathbb{mo}$-continuous operators. Its proof below is a modification 
of the proof of Theorem 2.1 from \cite{AS}.

\begin{theorem}\label{cont.isp-bdd}
{\em \ Let $T:E\to F$ be an operator between two Archimedean Riesz algebras.
\begin{enumerate}[(i)]
\item If $T$ is $\mathbb{mo}$-, $\mathbb{omo}$-, or $\mathbb{rmo}$-continuous then $T$ is $\mathbb{mo}$-bounded.
\item If $T$ is $\mathbb{m}_l\mathbb{o}$-, $\mathbb{om}_l\mathbb{o}$-, or $\mathbb{rm}_l\mathbb{o}$-continuous 
$($resp., $\mathbb{m}_r\mathbb{o}$-, $\mathbb{om}_r\mathbb{o}$-, or $\mathbb{rm}_r\mathbb{o}$--continuous$)$
then $T$ is $\mathbb{m}_l\mathbb{o}$-bounded $($resp., $\mathbb{m}_r\mathbb{o}$-bounded$)$.
\end{enumerate}}
\end{theorem}

\begin{proof}
(i) Let $T:E\to F$ be an $\mathbb{mo}$-continuous $($resp., $\mathbb{omo}$-continuous, $\mathbb{rmo}$-continuous$)$ 
operator between two Riesz algebras.
Take an arbitrary order interval $[0,b]\subseteq E$ and consider the directed
set ${\cal I}={\mathbb N}\times [0,b]$ with the lexicographical order \cite[Thm.2.1]{AS}. 
Take a net $x_{(n,y)}:=\frac{1}{n}y$ indexed by $\cal{I}$.
It follows from $0\le x_{(n,y)}\le\frac{1}{n}b\downarrow 0$
that $x_{(n,y)}\rc 0$. Then $x_{(n,y)}\oc 0$, $x_{(n,y)}\mrc 0$, and $x_{(n,y)}\moc 0$. 
If $T$ is $\mathbb{mo}$-continuous $($resp., $\mathbb{omo}$-, and $\mathbb{rmo}$-continuous$)$ then $Tx_{(n,y)}\moc 0$. Hence
$$
   |w\cdot T(x_{(n,y)})|\vee|T(x_{(n,y)})\cdot w|\oc 0 \ \ \ (\forall w\in F_+).
$$
Take any $w\in F_+$. Then there exists a net $z_\beta\downarrow 0$ in $F$ such that
for every $\beta$ there exists $(n_\beta,y_\beta)\in\cal{I}$ 
satisfying $|w\cdot T(x_{(n,y)})|\vee|T(x_{(n,y)})\cdot w|\le z_\beta$ 
for all $(m,y)\ge (n_\beta,y_\beta)$. Take any $z_\beta$. Then, in particular, for all $u\in [0,b]$,
$$
  \frac{1}{n_\beta+1}(|w\cdot Tu|\vee |(Tu)\cdot w|)=
  |w\cdot T(x_{(n_\beta+1,u)})| \vee |T(x_{(n_\beta+1,u)})\cdot w|\le z_\beta.
\eqno(4)
$$ 
By (4), $w\cdot T[0,b] \cup (T[0,b]) \cdot w\subseteq [-(n_\beta+1)z_\beta,(n_\beta+1)z_\beta]$ 
and since $w\in F_+$ is arbitrary, we conclude that $T$ is $\mathbb{mo}$-bounded.

(ii) The modification of the proof in $(i)$ for the "$\mathbb{m}_l\mathbb{o}$-"
and "$\mathbb{m}_r\mathbb{o}$-bounded case" is trivial.
\end{proof}

\section{$\mathbb{mo}$-Lebesgue, $\mathbb{mo}$-$KB$, and $\mathbb{mo}$-Levi operators}

In this section, we  undertake an attempt to adopt some of results of the recent paper  \cite{AlEG} to  $\mathbb{mo}$-convergence
in Riesz algebras. Especially, we investigate the domination problem for related operators.

\begin{definition}\label{order-to-topology}
{\em We say that an operator $T$
\begin{enumerate}
\item[$a)$]  
from a Riesz space $X$ to a Riesz algebra $F$ is 
{\em $\mathbb{mo}$-Lebes\-gue} if $Tx_\alpha\moc 0$ for every net $x_\alpha$ in $X$ such that $x_\alpha\downarrow 0$.
In particular,  every $\mathbb{omo}$-continuous operator is $\mathbb{mo}$-Lebesgue$;$
\item[$b)$] 
from a locally solid Riesz space $X=(X,\tau)$ to a Riesz algebra $F$ is an
{\em $\mathbb{mo}$-$KB$-ope\-ra\-tor} if, for every $\tau$-bounded increasing net $x_\alpha$ in $X_+$, there exists $x\in X$ such that $Tx_\alpha\moc Tx$$;$
\item[$c)$] 
from a locally solid Riesz space $X=(X,\tau)$ to a Riesz algebra $F$ is a
{\em quasi $\mathbb{mo}$-$KB$-ope\-ra\-tor} if, for every $\tau$-bounded increasing net $x_\alpha$ in $X_+$, $Tx_\alpha$ is an $\mathbb{mo}$-Cauchy net$;$
\item[$d)$]   
from a Riesz algebra $E$ to a Riesz algebra $F$ is an
{\em $\mathbb{mo}$-Levi operator}
if, for every $\mathbb{mo}$-bounded increasing net $x_\alpha$ in $E_+$,
there exists $x\in E$ such that $Tx_\alpha\moc Tx$$;$ 
\item[$e)$]   
from a Riesz algebra $E$ to a Riesz algebra $F$ is a {\em quasi $\mathbb{mo}$-Levi ope\-ra\-tor}
if, for every $\mathbb{mo}$-bounded increasing net $x_\alpha$ in $E_+$, the net 
$Tx_\alpha$ is $\mathbb{mo}$-Cauchy in $F$$.$ 
\end{enumerate}
The {\em $\mathbb{m}_l\mathbb{o}$-} {\em $\mathbb{m}_r\mathbb{o}$-Lebes\-gue operators},
the $\mathbb{m}_l\mathbb{o}$- and {\em $\mathbb{m}_r\mathbb{o}$-$KB$-operators},
the {\em quasi $\mathbb{m}_l\mathbb{o}$-$KB$}- and {\em quasi $\mathbb{m}_r\mathbb{o}$-$KB$-operators},
the {\em $\mathbb{m}_l\mathbb{o}$-} and {\em $\mathbb{m}_r\mathbb{o}$-Levi operators}, 
the {\em quasi $\mathbb{m}_l\mathbb{o}$-Levi} and {\em quasi $\mathbb{m}_r\mathbb{o}$-Levi operators} 
are defined analogiously.}
\end{definition}

Although it seems that the sequential versions of Definition~\ref{order-to-topology} and the adopted for 
$\mathbb{r}$-convergence modifications of this definition are also interesting, we do not study them in the present paper.

\begin{example}\label{Example5.1}
{\em (cf. \cite[Ex.3.1]{AlEG})
Let $E$ be the $f$-algebra of all bounded real functions on $[0,1]$ which differ from a constant on at most countable subset of $[0,1]$
equipped with the poinwise algebraic multiplication. 
Let $T:E \to E$ be an operator that assigns to each $f\in E$ the constant function 
$Tf$ on  $[0,1]$ such that the set $\{t\in [0,1] : f(t) \ne (Tf)(t)\}$ is at most countable. Then $T$ is a rank one operator which is  continuous in the $\sup$-norm on $E$. Consider the following net in $E$ indexed by finite subsets  of $[0,1]$: 
$$ 
   f_\alpha(t) = \left\{
   \begin{array}{ccc}
   1 &\text{ if } & t \not\in \alpha\\
   0 &\text{ if } & t \in \alpha
   \end{array}
   \right..
$$
Then $f_\alpha \downarrow 0$ in $E$ and so $f_\alpha \moc 0$. However, 
$Tf_\alpha \equiv \mathbb{1}_{[0,1]}$  for all $\alpha$,
and hence  $Tf_\alpha \moc  \mathbb{1}_{[0,1]} \ne 0$.
Therefore $T$ is neither $\mathbb{omo}$- nor $\mathbb{mo}$-Lebesgue.}
\end{example}

It is well known that each order bounded disjointness preserving operator $T$ between Riesz spaces $X$ and $Y$
has modulus $|T|$, and $|T||x|=|T|x||=|Tx|$ for all $x\in X$;
and there exist Riesz homomorphisms $R_1,R_2:X\to Y$ such that $T = R_1-R_2$. 

The next theorem is motivated by \cite[Thm.2.5]{AlEG} and has the similar proof. 

\begin{theorem}\label{KB disj pres dominated property}
Let $T$ be an order bounded disjointness preserving $\mathbb{mo}$-$KB$-operator 
from a locally solid Riesz space $(X,\tau)$ to a Riesz algebra $F$.
If $|S|\le|T|$ then $S$ is an $\mathbb{mo}$-$KB$-operator.
The similar result holds true for $\mathbb{m}_l\mathbb{o}$-$KB$  and $\mathbb{m}_r\mathbb{o}$-$KB$ operators.
\end{theorem}

\begin{proof}
We restrict ourselves to the case of $\mathbb{mo}$-$KB$-operators.
Take a $\tau$-bounded increasing net $x_\alpha$ in $X_+$. Then $T(x_\alpha-x)\moc 0$ for some $x\in X$. 
So, for every $u\in F_+$, $u\cdot |T(x_\alpha-x)|\oc 0$ and $|T(x_\alpha-x)| \cdot u \oc 0$, and hence
$$
   u\cdot|S(x_\alpha-x)|\le u\cdot|S||x_\alpha-x|\le u\cdot|T||x_\alpha-x|=u\cdot|Tx_\alpha-Tx|\oc 0;
$$ 
$$
   |S(x_\alpha-x)|\cdot u\le |S||x_\alpha-x|\cdot u\le |T||x_\alpha-x|\cdot u=|Tx_\alpha-Tx|\cdot u\oc 0.
$$ 
Thus $(u\cdot|Sx_\alpha-Sx|)\vee(|Sx_\alpha-Sx|\cdot u)\oc 0$ for every $u\in F_+$.
Therefore $Sx_\alpha\moc Sx$, and hence $S$ is an $\mathbb{mo}$-$KB$-operator.
\end{proof}

Since, for a  Riesz homomorphism $T$,  $0\le S\le T$ implies that $S$ is also a Riesz homomorphism, 
the next result follows from  Theorem \ref{KB disj pres dominated property}.

\begin{corollary}\label{KB lattice homomorphism dominated property}
Let $T$ be an $\mathbb{mo}$-$KB$ Riesz homomorphism 
from a locally solid Riesz space to a Riesz algebra.
Then every operator $S$ satisfying  $0\le S\le T$ is also an $\mathbb{mo}$-$KB$ Riesz homomorphism.
The similar result is true for $\mathbb{m}_l\mathbb{o}$-$KB$ and $\mathbb{m}_r\mathbb{o}$-$KB$ Riesz homomorphisms.
\end{corollary}

The following result is analogous to \cite[Thm.2.6]{AlEG}. 

\begin{theorem}\label{quasi $KB$ dominated property}
Let $T$ be a positive quasi $\mathbb{mo}$-$KB$ operator  
from a locally solid Riesz space $(X,\tau)$ to a Riesz algebra $F$.
Then every operator $S$ satisfying  $0\le S\le T$ is also a quasi $\mathbb{mo}$-$KB$ operator.
The similar result holds true for quasi $\mathbb{m}_l\mathbb{o}$-$KB$ and quasi $\mathbb{m}_r\mathbb{o}$-$KB$ operators.
\end{theorem}

\begin{proof}
Take an increasing $\tau$-bounded net $x_\alpha$ in $X_+$ and let $0\le S\le T$.
Then $Tx_\alpha\uparrow$, and since $T$ is quasi $\mathbb{mo}$-$KB$, the net $Tx_\alpha$ is $\mathbb{mo}$-Cauchy. 
Pick a $u\in F_+$. Then there exists a net $z_\beta\downarrow$ in $F$ such that, for each $\beta$,
there exists $\alpha_\beta$ such that 
$u\cdot |Tx_{\alpha_1}-Tx_{\alpha_2}|\le z_\beta$ for all $\alpha_1, \alpha_2\ge \alpha_\beta$.
Choosing $\alpha_1, \alpha_2\ge \alpha_\beta$ for a fixed $\alpha_\beta$ we have
$$
   u\cdot |Sx_{\alpha_1}-Sx_{\alpha_2}|\le 
   u\cdot |S(x_{\alpha_1}-x_{\alpha_\beta})|\le 
   u\cdot |T(x_{\alpha_1}-x_{\alpha_\beta})|\le z_\beta.
  \eqno(5)
$$
By (5), $u\cdot |Sx_{\alpha_1}-Sx_{\alpha_2}|\le z_\beta$ for all 
$\alpha_1, \alpha_2\ge \alpha_\beta$. Thus the net $Sx_\alpha$ is $\mathbb{m}_l\mathbb{o}$-Cauchy.
Hence, the operator $S$ is quasi $\mathbb{m}_l\mathbb{o}$-$KB$.
Similar argument shows that $S$ is quasi $\mathbb{m}_r\mathbb{o}$-$KB$.
Therefore, $S$ is a quasi $\mathbb{m}\mathbb{o}$-$KB$-operator.
\end{proof}

The following result is motivated by \cite[Thm.2.7]{AlEG}. 

\begin{theorem}\label{quasi Levi dominated property}
Let $T$ be a positive quasi $\mathbb{m}_l\mathbb{o}$-Levi ope\-ra\-tor from a Riesz algebra $E$ to a Riesz algebra $F$.
Then each operator $S:E\to F$ satisfying $0\le S\le T$ is also quasi $\mathbb{m}_l\mathbb{o}$-Levi.
The similar result holds true for quasi $\mathbb{m}_r\mathbb{o}$-Levi and $\mathbb{m}\mathbb{o}$-Levi operators.
\end{theorem}

\begin{proof}
Let $x_\alpha$ be an $\mathbb{m}_l\mathbb{o}$-bounded increasing net in $E_+$.
Then the net $Tx_\alpha$ is $\mathbb{m}_l\mathbb{o}$-Cauchy in $F$. Pick some $u\in F_+$.
There exists a net $z_\beta\downarrow 0$ in $F$ such that, for each $\beta$,
there exists $\alpha_\beta$ such that $u\cdot |Tx_{\alpha_1}-Tx_{\alpha_2}|\le z_\beta$ for all $\alpha_1, \alpha_2\ge \alpha_\beta$.
Pick $\alpha_1, \alpha_2\ge \alpha_\beta$ for a fixed $\alpha_\beta$. Then
$$
   u\cdot |Sx_{\alpha_1}-Sx_{\alpha_2}|\le u\cdot |S(x_{\alpha_1}-x_{\alpha_\beta})|\le 
   u\cdot |T(x_{\alpha_1}-x_{\alpha_\beta})| \le z_\beta.
\eqno(6)
$$
By (6),  $u\cdot |Sx_{\alpha_1}-Sx_{\alpha_2}|\le z_\beta$ for all $\alpha_1, \alpha_2\ge \alpha_\beta$. 
Since $u\in F_+$ is arbitrary, the net $Sx_\alpha$ is $\mathbb{m}_l\mathbb{o}$-Cauchy.
Hence, $S$ is quasi $\mathbb{m}_l\mathbb{o}$-Levi. Analogously, if $T$ is quasi $\mathbb{m}_r\mathbb{o}$-Levi 
or $\mathbb{m}\mathbb{o}$-Levi then $S$ has the same property.
\end{proof}

We conclude the section with short discussion of extensions of operators introduced in Definition~\ref{order-to-topology} 
to the second order duals. In what follows we assume that the first order dual of any Riesz space under the consideration separates its points. We recall that, under these conditions, any Riesz space  $X$ is embedded in its second order dual $X''$ via the mapping $x \stackrel{i}{\to} x''$, where $x''(z)=z(x)$ for all $z\in X'$. 
If $(E, \cdot)$ is a Riesz algebra, then $E''$ is again Riesz algebra with respect to the {\em Arens multiplication} \cite{HuP}. In all cases of an operator $T$ considered in Definition~\ref{order-to-topology}, i.e. $T:X\to F$
and $T:E\to F$, the second  dual $T''$ is acting between spaces of the same types. Therefore
it is natural to ask whether or not, for an  $\mathbb{mo}$-Lebes\-gue, $\mathbb{mo}$-$KB$-, and $\mathbb{mo}$-Levi operator $T$, the operator $T''$ is $\mathbb{mo}$-Lebes\-gue, $\mathbb{mo}$-$KB$-, and $\mathbb{mo}$-Levi respectively.  The authors do not know examples of $T$, for which the answer to the above question is negative.

{\tiny 

}
\end{document}